\documentclass [12pt]{amsart}
\usepackage{amssymb,amsmath,amsthm}
\newtheorem{theorem}{Theorem}
\newtheorem{lemma}{Lemma}
\newtheorem{proposition}{Proposition}
\newtheorem{definition}{Definition}

\newcommand{\ad}{\,\mathrm{ad}\,}

\newcommand{\GL}{\,\mathrm{GL}\,}
\newcommand{\SL}{\,\mathrm{SL}\,}

\newcommand{\diag}{\,\mathrm{diag}\,}
\begin{document}

\begin{center}
{\Large {\bf Automorphisms and isomorphisms of Chevalley groups of type $G_2$\\

\bigskip

 over local rings with $1/2$ and $1/3$\footnote{The
work is supported by the Russian President grant MK-2530.2008.1 and
by the grant of Russian Fond of Basic Research 08-01-00693.} }}

\bigskip
\bigskip

{\large \bf E.~I.~Bunina}
\bigskip

{ \bf M.V. Lomonosov Moscow State University}

\bigskip

{ Russia, 119992, Moscow, Leninskie Gory, Main Building of MSU, Faculty of Mechanics and Mathematics, Department of Higher Algebra}

\bigskip

{ email address: helenbunina@yandex.ru}

\end{center}
\bigskip

\begin{center}

{\bf Abstract.}

\end{center}

We prove that every isomorphism of Chevalley groups of type $G_2$ over commutative local rings  with~$1/2$ and $1/3$ is standard, i.\,e., it is a composition of  a ring isomorphism and an inner automorphism.
\medskip

{\bf Key words:} Chevalley groups, local rings, isomorphisms and automorphisms

\bigskip

\section*{Introduction}\leavevmode

An associative commutative ring $R$ with a unit is called  \emph{local},
if it contains exactly one maximal ideal (that coincides with the radical of~$R$). Equivalently, the set of all non-invertible
elements of~$R$ is an ideal.

We describe automorphisms of Chevalley groups of type $G_2$ over local rings with~$1/2$ and $1/3$. Note that for the root system $G_2$ there exists only one weight lattice, that is simultaneously universal and adjoint, therefore for every ring~$R$ there exists a unique Chevalley group of type $G_2$, that is
 $G(R)=G_{\ad}(G_2,R)$. Over local rings universal Chevalley groups coincide with their elementary subgroups, consequently the Chevalley group $G(R)$ is also an elementary Chevalley group.

Theorem~1 for the root systems $A_l, D_l, $ and $E_l$ was obtained by the author in~\cite{ravnyekorni}, in~\cite{normalizers} all automorphisms of Chevalley groups of given types over local rings with~$1/2$ were described. Theorem~1 for the root systems $B_2$ and $G_2$ is obtained in the paper~\cite{korni2}, but we repeat it here for the root system~$G_2$ with an easier proof.

Similar results for Chevalley groups over fields were proved
 by R.\,Steinberg~\cite{Stb1} for the finite case and by J.\,Humphreys~\cite{H} for the infinite case. Many papers were devoted
to description of automorphisms of Chevalley groups over different
commutative rings, we can mention here the papers of
Borel--Tits~\cite{v22}, Carter--Chen~Yu~\cite{v24}, 
Chen~Yu~\cite{v25}--\cite{v29}, A.\,Klyachko~\cite{Klyachko}.
 E.\,Abe~\cite{Abe_OSN} proved that all automorphisms of Chevalley groups under Noetherian  rings with~$1/2$ are standard.

The case
$A_l$ was completely studied by the papers of
W.C.~Waterhouse~\cite{v46}, V.M.~Petechuk~\cite{v12},  Fuan Li and
Zunxian Li~\cite{v37}, and also for rings without~$1/2$. The paper
of I.Z.\,Golubchik and A.V.~Mikhalev~\cite{v8} covers the
case~$C_l$, that is not considered in the present paper. Automorphisms and isomorphisms of general linear groups over arbitrary associative rings were described by E.I.~Zelmanov in~\cite{v11} and by I.Z.~Golubchik, A.V.~Mikhalev in~\cite{GolMikh1}.

 We generalize some methods of V.M.~Petechuk~\cite{Petechuk1} to prove Theorem~1.

The author is thankful to N.A.\,Vavilov,  A.A.\,Klyachko,
A.V.\,Mikhalev for valuable advices, remarks and discussions.

\section{Definitions and main theorems.}\leavevmode

 We fix the root system~$\Phi$ of the type $G_2$ (detailed texts about root
systems and their properties can be found in the books
\cite{Hamfris}, \cite{Burbaki}). Let $e_1,e_2,e_3$
be an orthonorm basis of the space $\mathbb R^3$. Then we  numerate the roots of $G_2$
as follows:
$$
\alpha_1= e_1-e_2, \alpha_2=-2e_1+e_2+e_3
$$
are simple roots;
\begin{align*}
\alpha_3=\alpha_1+\alpha_2= e_3-e_1,\\
\alpha_4= 2\alpha_1+\alpha_2=e_3-e_2,\\
\alpha_5= 3\alpha_1+\alpha_2=e_1+e_3-2e_2,\\
\alpha_6=3\alpha_1+2\alpha_2=2e_3-e_1-e_2
\end{align*}
are other positive roots.

Suppose now that we have a
semisimple complex Lie algebra~$\mathcal L$ of type $G_2$ with
Cartan subalgebra~$\mathcal H$ (detailed information about
semisimple Lie algebras can be found in the book~\cite{Hamfris}).

 Then in the algebra $\mathcal L$ we can choose a \emph{Chevalley basis}
  $\{ h_i\mid i=1,2; x_\alpha\mid \alpha\in \Phi\}$ so that for
every two elements of this basis their commutator is an integral
linear combination of the elements of the same basis.

Namely,

1) $[h_i,h_j]=0;$

2) $[h_i,x_\alpha]=\langle \alpha_i,\alpha\rangle x_\alpha$;

3) if $\alpha=n_1\alpha_1+\dots+n_4\alpha_4$, then
$[x_{\alpha},x_{-\alpha}]=n_1h_1+\dots+n_4h_4$;

4) if $\alpha+\beta\notin \Phi$, then $[x_{\alpha},x_{\beta}]=0$;

5) if $\alpha+\beta\in \Phi$, and $\alpha,\beta$ are roots of the same length, then $[x_\alpha,x_\beta]=c x_{\alpha+\beta}$;

6) if $\alpha+\beta\in \Phi$, $\alpha$ is a long root, $\beta$
is a short root, then $[x_{\alpha},x_\beta]=a x_{\alpha+\beta}+b
x_{\alpha+2\beta}+\dots$.

Take now an arbitrary local ring with $1/2$ and $1/3$ and construct
an elementary adjoint Chevalley group of type $G_2$
over this ring (see, for example~\cite{Steinberg}). For our convenience we briefly put here the construction.

In the Chevalley basis of $\mathcal L$ all operators
$(x_\alpha)^k/k!$ for $k\in \mathbb N$ are written
as integral (nilpotent) matrices. An integral matrix also can be considered as
a matrix over an arbitrary commutative ring with~$1$.Let $R$ be
such a ring. Consider matrices $n\times n$ over~$R$, matrices $(x_\alpha)^k/k!$ for
 $\alpha\in \Phi$, $k\in \mathbb N$ are included in $M_n(R)$.

Now consider automorphisms of the free module $R^n$ of the form
$$
\exp (tx_\alpha)=x_\alpha(t)=1+tx_\alpha+t^2 (x_\alpha)^2/2+\dots+
t^k (x_\alpha)^k/k!+\dots
$$
Since all matrices $x_\alpha$ are nilpotent, we have that this
series is finite. Automorphisms $x_\alpha(t)$ are called
\emph{elementary root elements}. The subgroup in $Aut(R^n)$,
generated by all $x_\alpha(t)$, $\alpha\in \Phi$, $t\in R$, is
called an \emph{elementary adjoint Chevalley group} (notation:
$E_{\ad}(\Phi,R)=E_{\ad}(R)$).

In an elementary Chevalley group there are the following important elements:

--- $w_\alpha(t)=x_\alpha(t) x_{-\alpha}(-t^{-1})x_\alpha(t)$, $\alpha\in \Phi$,
$t\in R^*$;

--- $h_\alpha (t) = w_\alpha(t) w_\alpha(1)^{-1}$.

The action of $x_\alpha(t)$ on the Chevalley basis is described in
\cite{v23}, \cite{VavPlotk1}, we write it below.

Over local rings for the root system $G_2$
all Chevalley groups coincide with elementary adjoint Chevalley groups $E_{\ad}(R)$,
 therefore we do not introduce Chevalley
groups themselves in this paper.

We will work with two types of standard automorphisms of  Chevalley groups
 $G(R)$ and $G(S)$ and with one unusual, ``temporary'' type of automorphisms.

{\bf Ring isomorphisms.} Let $\rho: R\to S$ be an isomorphism of rings.
The mapping $x\mapsto \rho (x)$ from $G(R)$ onto $G(S)$
is an isomorphism of the groups  $G(R)$ and $G(S)$, it is denoted by the same letter~$\rho$ and called a
  \emph{ring isomorphism} of the groups~$G(R)$ and~$G(S)$. Note that for all
$\alpha\in \Phi$ and $t\in R$ the element $x_\alpha(t)$ is mapped to $x_\alpha(\rho(t))$.

{\bf Inner automorphisms.} Let $g\in G(R)$ be an element of a Chevalley group under consideration.
Conjugation of the group $G(R)$ with the element~$g$ is an automorphism of  $G(R)$, that is denoted by~$i_g$
and is called an  \emph{inner automorphism} of~$G(R)$.

These two types of automorphisms are called \emph{standard}. There are  central and graph automorphisms, which are also standard, but in our case (root system $G_2$) they can not appear. Therefore we say that an isomorphism of groups $G(R)$ and $G(S)$ is standard, if it is a composition of two introduced types of isomorphisms.

Besides that, we need also to introduce  temporarily one more type of automorphisms:

{\bf Automorphisms--conjugations.} Let $V$ be a representation space of the Chevalley group $G(R)$, $C\in \GL(V)$ be  a matrix from the normalizer of $G(R)$:
$$
C G(R) C^{-1}= G (R).
$$
 Then the mapping $x\mapsto CxC^{-1}$ from $G(R)$ onto itself is an automorphism of the Chevalley group, which  is denoted by $i_С$ and is called an \emph{automorphism--conjugation} of~$G(R)$,
\emph{induced by the element}~$C$ of the group~$\GL(V)$.

\smallskip

In Section~5 we will prove that in our case all automorphisms--conjugations are inner, but the first step is the proof of the following theorem:

\begin{theorem}\label{first} Suppose that $G(R)=G(\Phi,R)$ and $G(S)=G(\Phi,S)$ are Chevalley groups of type  $G_2$, $R$, $S$ are commutative local rings with~$1/2$ and $1/3$. Then every isomorphism of the groups $G(R)$ and $G(S)$ is a composition of a ring isomorphism and an automorphism--conjugation.
\end{theorem}

Sections 2--4 are  devoted to the proof of Theorem~1.

\section{Changing the initial isomorphism to a special isomorphism.}\leavevmode

In this section we use some arguments from the paper~\cite{Petechuk1}.

\begin{definition} \emph{By $\GL_n(R,J)$ we denote the subgroup of such matrices~$A$ from $\GL_n(R)$, that satisfy $A-E\in
M_n(J)$, where $J$ is the radical of~$R$.} \end{definition}

\begin{proposition}\label{p1_0} By an arbitrary isomorphism $\varphi$ between Chevalley groups $G(R)$ and $G(S)$
we can construct an isomorphism $\varphi'= i_{g^{-1}} \varphi$, $g\in GL_n(S)$, of the group
 $G(R)\subset \GL_n(R)$ onto some subgroup $\GL_n(S)$,
with the property that any matrix $A\in G(R)$ with elements of the subring of~$R$, generated by~$1$,
is mapped under~$\varphi'$ to the matrix from~$A\cdot \GL_n(S,J_S)$.
\end{proposition}
\begin{proof} Let $J_R$ be the maximal ideal (radical) of~$R$, $k$ the residue field $R/J_R$. Then the group $G(R,J_R)$ generated by all $x_{\alpha}(t)$, $\alpha\in \Phi$, $t\in J_R$, is the greatest normal proper subgroup in~$G(R)$ (see~\cite{Abe1}). Therefore under the action of~$\varphi$ the group $G(R,J_R)$ is mapped to $G(S,J_S)$.

By this reason the isomorphism $$ \varphi: G (R)\to G(S) $$ induces  an isomorphism $$ \overline \varphi: G
(R)/G(R,J_R)=G (R/J_R)\to G(S/J_S). $$ The groups $G(R/J_R)$ and $G(S/J_S)$ are Chevalley groups over fields, so that the isomorphism
$\overline \varphi$ is standard, i.\,e., it is $$ \overline \varphi =  i_{\overline g} \overline \rho,\quad \overline
g\in G(S/J_S)\quad \text{ (see~\cite{Steinberg}, \S\,10).} $$
 Clear that there exists a matrix $g\in \GL_n(S)$ such that its image under factorization   $S$ by~$J_S$ is~$\overline
g$. Note that it is not necessary  $g\in N(G(S))$.

Consider the mapping $\varphi'= i_{g^{-1}} \varphi$. It is an isomorphism of the group
 $G(R)\subset \GL_n(R)$ onto some subgroup in $\GL_n(S)$,
with the property that its image under factorization the rings by their radicals is the isomorphism $\overline \rho$.

Since the isomorphism  $\overline \rho$ acts identically on matrices with all elements generated by the unit of~$k$, we have that ane matrix $A\in G(R)$ with elements from the subring of~$R$, generated by~$1$, is mapped under the action of~$\varphi'$ to some matrix from the set~$A\cdot \GL_n(S,J_S)$.
 \end{proof}

Let $a\in G (R)$, $a^2=1$. Then the element $e=\frac{1}{2} (1+a)$ is an idempotent of the ring $M_n(R)$. This idempotent $e$ defines a decomposition of the free $R$-module $V=R^n$: $$ V=eV\oplus (1-e)V=V_0\oplus V_1 $$ (the modules $V_0$,
$V_1$ are free, because every projective module over a local ring is free). Let $\overline V=\overline V_0 \oplus
\overline V_1$ be a decomposition of the $k$-module~$\overline V$ with respect to~$\overline a$, and
 $\overline e=\frac{1}{2} (1+\overline a)$.

Then we have

\begin{proposition}\label{pr1_1} Modules \emph{(}subspaces\emph{)}
 $\overline V_0$, $\overline V_1$ are the images of $V_0$, $V_1$ under factorization  by~$J$.
\end{proposition} \begin{proof} Denote the images of $V_0$, $V_1$ under factorization $R$ by~$J$ by $\widetilde V_0$, $\widetilde
V_1$, respectively. Since $V_0=\{ x\in V\mid ex=x\},$ $V_1= \{ x\in V\mid ex=0\},$  we have
 $\overline e(\overline x)=\frac{1}{2}(1+\overline a)(\overline x)=\frac{1}{2}
(1+\overline a(\overline x))=\frac{1}{2}(1+\overline{a(x)})=\overline{e(x)}$. Then $\widetilde V_0\subseteq \overline
V_0$, $\widetilde V_1\subseteq \overline V_1$.

Let $x=x_0+x_1$, $x_0\in V_0$, $x_1\in V_1$. Then $\overline e(\overline x)=\overline e(\overline x_0)+\overline e
(\overline x_1)=\overline x_0$. If $\overline x\in \widetilde V_0$, then $\overline x=\overline x_0$.
\end{proof}

 Let
$b=\varphi'(a)$. Then $b^2=1$ and $b$ is equivalent to $a$ modulo~$J$.

\begin{proposition}\label{pr1_2} Suppose that $ a\in G(R)$, $b\in G(S)$ $a^2=b^2=E$, $a$ is a matrix with elements from the subring of~$R$, generated by the unit, $b$ and $a$ are equivalent modulo~$J_S$ \emph{(}since all elements of $a$ are generated by unit, we can consider it as an element of $M_n(S)$\emph{)}, $V=V_0\oplus V_1$ is a decomposition of~$V$
with respect to~$a$, $V=V_0'\oplus V_1'$ is decomposition of~$V$ with respect to~$b$. Then $\dim V_0'=\dim V_0$, $\dim V_1'=\dim
V_1$.
\end{proposition}

\begin{proof} We have an  $S$-basis of~$V$ $\{ e_1,\dots,e_n\}$ such that $\{ e_1,\dots,e_k\}\subset V_0$, $\{
e_{k+1},\dots, e_n\}\subset V_1$. Clear that $$ \overline a \overline e_i=\overline{ae_i}=\overline {(\sum_{j=1}^n a_{ij}
e_j)}= \sum_{j=1}^n \overline a_{ij} \overline e_j. $$ Let $\overline V=\overline V_0\oplus \overline V_1$, $\overline
V = \overline V_0'\oplus \overline V_1'$ be decompositions of $k=S/J_S$-module (space)~$\overline V$ with respect to $\overline
a$ and $ \overline b$. Clear that $\overline V_0= \overline V_0'$, $\overline V_1 =\overline V_1'$. Therefore, by Proposition~\ref{pr1_1}, the images of $V_0$ and $V_0'$, $V_1$ and $V_1'$ under factorization by~$J_S$ coincide. Take such
$\{ f_1,\dots, f_k\}\subset V_0'$, $\{ f_{k+1},\dots, f_n\}\subset V_1'$, that $\overline f_i=\overline e_i$,
$i=1,\dots,n$. Since a matrix that maps the basis $\{ e_1,\dots, e_n\}$ to $\{ f_1,\dots, f_n\}$ is invertible (it is equivalent to the unit matrix modulo~$J_S$), we have that $\{ f_1,\dots, f_n\}$ is an $S$-basis of~$V$. Cleat that $\{ f_1,\dots, f_k\}$
is an $S$-basis of $V_0'$, $\{ v_{k+1},\dots, v_n\}$ is an   $S$-basis of $V_1'$.
\end{proof}

\section{Images of~$w_{\alpha_i}$}

Consider Chevalley groups $G(R)$ and $G(S)$  of type $G_2$, their adjoint representations in the groups  $\GL_{14}(R)$ and $\GL_{14}(S)$, in bases of weight vectors $v_1=x_{\alpha_1},
v_{-1}=x_{-\alpha_1}, \dots, v_6=x_{\alpha_6}, v_{-6}=x_{-\alpha_{6}}, V_1=h_{1},V_2=h_{2}$, corresponding to the Chevalley basis of~$G_2$.

We suppose that by the isomorphism  $\varphi$ we constructed the isomorphism  $\varphi'= i_{g^{-1}} \varphi$, described in the previous section. Recall that it is an isomorphism of the group
 $G(R)\subset \GL_n(R)$ on  some subgroup of $\GL_n(S)$,
with the property that its image under factorization ring by their radical is a ring isomorphism $\overline
\rho$.

Consider the matrices $h_{\alpha_1}(-1), h_{\alpha_2}(-1)$  in $G(R)$. They are
 \begin{align*}
 h_{\alpha_1}(-1)&=\diag [1,1,-1,-1,-1,-1,-1,-1,-1,-1,1,1,1,1],\\
h_{\alpha_2}(-1)&=\diag[-1,-1,1,1,-1,-1,1,1,-1,-1,-1,-1,1,1] .\end{align*}

By Proposition~\ref{pr1_2} we know that every matrix $h_i=\varphi'(h_{\alpha_i}(-1))$ in some basis is diagonal with $\pm 1$ on its diagonal, and
the number of  $1$ and
$-1$ coincides with their number for the matrix $h_{\alpha_i}(-1)$.
Since $h_1$ and $h_2$ commute, there exists a basis, where  $h_1$ and $h_2$ have the same form as $h_{\alpha_1}(-1)$ and $h_{\alpha_2}(-1)$. Suppose that we came to this basis with the help of the matrix~$g_1$. Clear that
 $g_1\in \GL_n(S,J_S)$. Consider the mapping
 $\varphi_1=i_{g_1}^{-1} \varphi'$. It is also an isomorphism of the group
 $G(R)$ onto some subgroup of $\GL_n(S)$ such that its image under factorization rings by their radicals is~$\overline \rho$, and
$\varphi_1(h_{\alpha_i}(-1))=h_{\alpha_i}(-1)$ for  $i=1,2$.

Instead of $\varphi'$ we now consider the isomorphism~$\varphi_1$.

\bigskip

{\bf Remark 1.} Every element $w_i=w_{\alpha_i}(1)$ moves by conjugation  $h_i$ to each other, therefore its image has a block-monomial form. In particular, this image can be rewritten as a block-diagonal matrix, where the first block is $12\times 12$, and the second is $2\times 2$.
\bigskip

Consider the first basis vector after the last basis change in $\GL_{14}(S)$. Denote it by~$e$. The Weil group  $W$
acts transitively on the set of roots of the same length, therefore for every
root~$\alpha_i$ of the same length as the first one,
there exists such $w^{(\alpha_i)}\in W$, that $w^{(\alpha_i)}
\alpha_1=\alpha_i$. Similarly, all roots of the second length are also conjugate under the action
of~$W$. Let $\alpha_k$ be the first root of the length that is not equal to the length of~$\alpha_1$, and let $f$ be
the $k$-th basis vector after the last basis change. If $\alpha_j$ is a root conjugate to $\alpha_k$, then let us denote by $w_{(\alpha_j)}$ an element of~$W$ such that
$w_{(\alpha_j)} \alpha_k=\alpha_j$. Consider now the basis
$e_1,\dots, e_{14}$, where $e_1=e$, $e_k=f$,
for $1< i\leqslant 12$ either $e_i=\varphi_1(w^{(\alpha_i)})e$, or
$e_i=\varphi_1(w_{(\alpha_i)})f$ (it depends of the length of $\alpha_k$);  we do not move $e_{13}$, $e_{14}$. Clear that the matrix of this basis change is equivalent to the unit modulo radical. Therefore the obtained set of vectors also is a basis.

Clear that the matrices  $\varphi_1(w_i)$ ($i=1,2$) on the basis part
 $\{ e_1,\dots,e_{12}\}$
coincide with the matrices for $w_i$ in the initial basis of weight vectors.
Since $h_i(-1)$ are squares of $w_i$, then there images are not changed in the new basis.

Besides, we know (Remark 1) that every matrix $\varphi_1(w_i)$
is block-diagonal up to decomposition of basis in the first  $12$
and last $2$ elements. Therefore the last part of basis consisting of $2$ elements, can be changed independently.

Denote the matrices $w_i$ and $\varphi_1(w_i)$ on this basis part by $\widetilde w_i$ and $\widetilde{\varphi_1(w_i)}$
respectively, and the $2$-generated module generated by $e_{13}$ and $e_{14}$, by $\widetilde V$.

\begin{lemma}\label{l3_3} For the root system  $G_2$ there exists such a basis that $\widetilde{\varphi_1(w_1)}$ and
$\widetilde{\varphi_1(w_2)}$ in this basis are $\widetilde{w_1}$ and $\widetilde{w_2}$. Namely, they are equal to  $$
\begin{pmatrix} -1& 3\\ 0& 1 \end{pmatrix} \text{ and } \begin{pmatrix} 1& 0\\ 1& -1 \end{pmatrix}. $$
\end{lemma}

\begin{proof}
Since $\widetilde w_1$ is an involution and $\widetilde V_1^1$ has dimension~$1$, there exists a basis $\{ e_1,e_2\}$, where $\widetilde{\varphi_1(w_1)}$ has the form $\diag [-1,1]$. In the basis  $\{ e_1, e_2-3/2e_1\}$ the matrix has the obtained form for~$G_2$.

Let the matrix $\widetilde{\varphi_1(w_2)}$ in this new basis be $$ \begin{pmatrix} a& b\\ c& d \end{pmatrix}. $$
Make the basis change with the help of $$ \begin{pmatrix} 1& (1-a)/c\\ 0& 1+\frac{2(1-a)}{c} \end{pmatrix} $$
(it is possible, since $c$ is equivalent to the unit modulo radical). Under such basis change the matrix $\widetilde{\varphi_1(w_1)}$ is not change and the matrix $\widetilde{\varphi_1(w_2)}$ becomes $$
\begin{pmatrix} 1& b'\\ c'& d' \end{pmatrix}. $$ Since this last matrix is an involution, we have $c'(1+d')=0$, $1+b'c'=1$.
Therefore, $d'=-1$, $b'=0$.

 Besides,
 $$ \left(\begin{pmatrix} -1& 1\\ 0& 1 \end{pmatrix} \begin{pmatrix} 1& 0\\ c'& -1 \end{pmatrix}\right)^3
= \left(\begin{pmatrix} 1& 0\\ c'& -1 \end{pmatrix} \begin{pmatrix} -1& 1\\ 0& 1 \end{pmatrix}\right)^3, $$ consequently
$3c'(3c'-1)(c'-1)=0$. Since $c'\equiv 1\mod J$, $3c'-1\equiv 2 \mod J$, $3\in R^*$, we have $c'-1=0$.
 \end{proof}

Therefore we can now come from the isomorphism~$\varphi_1$ under consideration to an isomorphism $\varphi_2$ with all properties of $\varphi_1$ and such that $\varphi_2(w_1)=w_1$, $\varphi_2(w_2)=w_2'$.

We suppose now that an isomorphism~$\varphi_2$ with all these properties is given.

\section{Images of $x_{\alpha_i}(1)$ and diagonal matrices.}\leavevmode

Recall that \\ $w_1=-e_{1,2}-e_{2,1}+e_{3,9}+e_{4,10}-e_{9,3}-e_{10,4}+e_{5,7}+e_{6,8}-e_{7,5}-e_{8,6}+e_{11,,11}+e_{12,12}
-e_{13,13}+e_{14,14}+3e_{13,14}$;\\
$w_2=-e_{3,4}-e_{4,3}+e_{1,5}+e_{2,6}-e_{5,1}-e_{6,2}+e_{7,7}+e_{8,8}+e_{9,11}+e_{10,12}-e_{11,9}-e_{12,10}+
e_{13,13}-e_{14,14}+e_{14,13}$;\\
$x_{\alpha_1}(1)=E-e_{1,2}-2e_{1,13}+3e_{1,14}-e_{4,6}-e_{4,8}-e_{4,10}+3e_{5,3}+2e_{6,8}+3e_{6,10}-3e_{7,3}-2e_{7,5}
+3e_{8,10}+e_{9,3}+e_{9,5}-e_{9,7}+e_{13,2}$;\\
$x_{\alpha_2}(1)=E+e_{2,6}-e_{3,4}+e_{3,13}-2e_{3,14}-e_{5,1}+e_{10,12}-e_{11,9}+e_{14,4}$.

 Since  $x_1=\varphi_2(x_{\alpha_1}(1))$ commutes with
$h_{\alpha_1}(-1)$ and with $w_{3\alpha_1+2\alpha_2}=w_2w_1w_2w_1^{-1}{w_2^{-1}}$, we have that $x_1$  can be decomposed to the blocks $\{ v_1,v_{-1}, v_{6}, v_{-6}, V_1,V_2\}$ and $\{ v_2,v_{-2},v_3, v_{-3}, v_4,v_{-4}, v_5, v_{-5}\}$; on the first block the matrix is
$$
\begin{pmatrix}
y_1& y_2& y_3& -y_3& y_4& -\frac{3y_4}{2}\\
y_5& y_6& y_7& -y_7& y_8& -\frac{3y_8}{2}\\
y_9& y_{10}& y_{11}& y_{12}& y_{13}& y_{14}\\
-y_9& -y_{10}& y_{12}& y_{11}& -y_{13}& 3y_{13}+y_{14}\\
y_{15}& y_{16}& y_{17}& y_{17}+3y_{19}& y_{18}& \frac{3(y_{20}-y_{18})}{2}\\
0& 0& y_{19}& y_{19}& 0& y_{20}
\end{pmatrix};
$$
on the second block it is
$$
\begin{pmatrix}
y_{21}& y_{22}& y_{23}& y_{24}& -y_{25}& -y_{26}& -y_{27}& -y_{28}\\
y_{29}& y_{30}& y_{31}& y_{32}& -y_{33}& -y_{34}& -y_{35}& -y_{36}\\
y_{37}& y_{38}& y_{39}& y_{40}& -y_{41}& -y_{42}& -y_{43}& -y_{44}\\
y_{45}& y_{46}& y_{47}& y_{48}& -y_{49}& -y_{50}& -y_{51}& -y_{52}\\
y_{52}& y_{51}& y_{50}& y_{49}& y_{48}& y_{47}& y_{46}& y_{45}\\
y_{44}& y_{43}& y_{42}& y_{41}& y_{40}& y_{39}& y_{38}& y_{37}\\
y_{36}& y_{35}& y_{34}& y_{33}& y_{32}& y_{31}& y_{30}& y_{29}\\
y_{28}& y_{27}& y_{26}& y_{25}& y_{24}& y_{23}& y_{22}& y_{21}
\end{pmatrix}.
$$

 Similalrly, since $x_2=\varphi_2(x_{\alpha_2}(1))$ commutes with
$h_{\alpha_2}(-1)$ and $w_{2\alpha_1+\alpha_2}=w_1w_2w_1{w_2^{-1}}w_1^{-1}$, we have decomposition of~$x_2$ on the blocks $\{ v_1,v_{-1},v_3, v_{-3}, v_5,v_{-5},v_6,v_{-6}\}$ and $\{v_2,v_{-2}, v_4,v_{-4}, V_1,V_2\}$; on the first block the matrix is
$$
\begin{pmatrix}
z_1& z_2& z_3& z_4& z_5& z_6& z_7& z_8\\
z_9& z_{10}& z_{11}& z_{12}& z_{13}& z_{14}& z_{15}& z_{16}\\
-z_{12}& -z_{11}& z_{10}& z_9& -z_{16}& -z_{15}& z_{13}& z_{14}\\
-z_4& -z_3& z_2& z_1& -z_8& -z_7& z_6& z_5\\
z_{17}& z_{18}& z_{19}& z_{20}& z_{21}& z_{22}& z_{23}& z_{24}\\
z_{25}& z_{26}& z_{27}& z_{28}& z_{29}& z_{30}& z_{31}& z_{32}\\
-z_{28}& -z_{27}& z_{26}& z_{25}& -z_{32}& -z_{31}& z_{30}& z_{29}\\
-z_{20}& -z_{19}& z_{18}& z_{17}& -z_{24}& -z_{23}& z_{22}& z_{21}
\end{pmatrix};
$$
on the second block it is
$$
\begin{pmatrix}
z_{33}& z_{34}& -z_{35}& z_{35}& z_{36}& 2z_{36}\\
z_{37}& z_{38}& -z_{39}& z_{39}& z_{40}& -2z_{40}\\
z_{41}& z_{42}& z_{43}& z_{44}& z_{45}& z_{46}\\
-z_{41}& -z_{42}& z_{44}& z_{43}& z_{45}+z_{46}& -z_{46}\\
0& 0& z_{47}+z_{48} & z_{47}+z_{48} & 2z_{49}+z_{50}& 0\\
z_{51}& z_{52}& z_{47}& z_{48}& z_{49}& z_{50}
\end{pmatrix}.
$$

Therefore we have $104$ variables $y_1,\dots, y_{52}, z_1,\dots, z_{52}$, where $y_1$, $y_6$, $y_{11}$, $y_{16}$, $y_{18}$, $y_{20}$, $y_{21}$, $y_{30}$, $y_{34}$, $y_{36}$, $y_{39}$, $y_{48}$, $z_1$, $z_{10}$, $z_{12}$, $z_{21}$, $z_{30}$, $z_{32}$, $z_{33}$, $z_{36}$, $z_{38}$,  $z_{43}$, $z_{50}$, $z_{52}$ are equivalent to~$1$ modulo radical,  $y_2$,  $y_{32}$, $z_{34}$ are equivalent to~$-1$, $y_4, y_{50}, $ are equivalent to~$-2$,  $y_{37}, -y_{52}$ are equivalent to~$3$,  all other elements are from the radical.

We apply step by step four basis changes, commuting with each other and with all matrices $w_{i}$. These changes are represented by matrices $C_1$, $C_2$, $C_3$, $C_4$. Matrices $C_1$ and $C_2$ are block-diagonal, with  $2\times 2$ blocks. On all  $2\times 2$ blocks, corresponding to short roots, the matrix  $C_1$ is unit, on all  $2\times 2$ blocks, corresponding to long roots, it is
$$
\begin{pmatrix}
1& -z_{51}/z_{52}\\
-z_{51}/z_{52}& 1
\end{pmatrix}.
$$
On the last block it is unit.

Similarly,  $C_2$ is unit on the blocks corresponding to long roots, and on the last block. On the blocks corresponding to the short roots, it is
$$
\begin{pmatrix}
1& -y_{15}/y_{16}\\
-y_{15}/y_{16}& 1
\end{pmatrix}.
$$
Matrices $C_3$ and $C_4$ are diagonal, identical on the last  block, the matrix $C_3$ is identical on all places, corresponding to short root, and scalar with multiplier~$a$ on all places corresponding to long roots. In the contrary, the matrix  $C_4$, is identical on all places, corresponding to long roots, and is scalar with multiplier~$b$ on all places, corresponding to short roots.

Since all these four matrices commutes with all $w_{i}$, $i=1,2$, then after basis change with any of these matrices all conditions for elements $x_1$ and $x_2$  still hold.

At the beginning we apply basis changes with the matrices $C_1$ and $C_2$. After that new $y_{15}$ in the matrix $x_1$ and $z_{51}$ in the matrix $x_2$ are equal to zero (for the convenience of notations we do not change names of variables). Then we choose  $a=-1/z_{52}$ (it is new $z_{52}$) and apply the third basis change. After it $z_{52}$ in the matrix $x_1$ becomes to be $1$. Clear that $z_{51}$ is still zero.

Finally, apply the last basis change with $b=1/y_{16}$ (where $y_{16}$ is the last one, obtained after all previous changes). We have that $y_{15}, z_{51}, z_{52}$ are not changed, and $y_{16}$ is now $1$.

Now we can suppose that $y_{15}=0, y_{16}=1, z_{51}=0, z_{52}=1$, and we have $100$ variables.

 Introduce $x_{1+2}=\varphi_2(x_{\alpha_1+\alpha_2}(1))=w_2'x_1{w_2^{-1}}'$, $x_{1+1+2}=\varphi_2(x_{2\alpha_1+\alpha_2}(1))=w_1x_{1+2}{w_1^{-1}}$, $x_{1+1+1+2}=\varphi_2(x_{3\alpha_1+\alpha_2}(1))=w_1x_2{w_1^{-1}}$, $x_{1+1+1+2+2}=\varphi_2(x_{3\alpha_1+2\alpha_2}(1))=w_2'x_{1+1+1+2}{w_2^{-1}}'$.

Now we will use the following conditions, that are true for elements $w_i$ and $x_i$:
\begin{align*}
Con1&=(x_2x_{1+2}=x_{1+2}x_2),\\
Con2&=(h_1x_2h_1x_2=E);\\
Con3&=(h_2x_1h_2x_1=E);\\
Con4&=(x_2x_{1+1+2}=x_{1+1+2}x_2);\\
Con5&=(x_2x_{1+1+1+2}=x_{1+1+1+2+2}x_{1+1+1+2}x_2);\\
Con6&=(x_{1+1+1+2}x_1=x_1x_{1+1+1+2};\\
Con7&=(x_1x_{1+1+2}=x_{1+1+1+2}^3x_{1+1+2}x_1);\\
Con8&=(w_1^3=x_1w_1x_1w_1^3x_1).
\end{align*}

Note that every matrix condition is  $196$
polynomial identities, where all polynomials have integer coefficients and depends of $y_i, z_j$. Temporarily numerate all variables as $v_1,\dots,v_{100}$.

Suppose that one of our polynomials can be rewritten in the form
$$
(v_{k_0}-\overline v_{k_0})A+v_1B_1+\dots+v_{k_0-1}B_{k_0-1}+v_{k_0+1}B_{k_0+1}+\dots+v_{100}B_{100}=0,
$$
where $\overline  v_{k_0}$ is such an integer number that is equivalent to $v_{k_0}$ modulo radical, the polynomial $A$ in invertible modulo radical, $B_i$ are some polynomials (the variable $v_{k_0}$ can enter in all polynomial and also in~$A$). Then
$$
v_{k_0}=-\frac{v_1B_1+\dots+v_{k_0-1}B_{k_0-1}+v_{k_0+1}B_{k_0+1}+\dots+v_{100}B_{100}}{A},
$$
we can substitute  the expression for $v_{k_0}$ in all other polynomial conditions. If we can choose $100$ such conditions that on every step we except one new variable, then on the last step we obtain the expression
$$
(v_{k_{100}}-\overline v_{k_{100}})C=0,
$$
where $C$ is some rational expression of variables
$v_1,\dots, v_{100}$, invertible modulo radical. Therefore,
we can say that $v_{k_{100}}=\overline v_{k_{100}}$, and consequently all other variables are equal to the integer numbers equivalent them modulo radical. The existence of the obtained $100$
conditions is equivalent to the existence of such $100$ conditions that
the square matrix consisting of all coefficients of these conditions modulo radical has an invertible determinant.

Since it is very complicated to write a matrix $100\times 100$, we will  sequentially take the obtained equations, but for  simplicity write coefficients $A$ and $B_i$ modulo radical (in the result these coefficients are just numbers $0$,  $\pm 1$, $\pm 2$, $\pm 3$).

We write below how the variables are expressed from the conditions (in brackets we write the number of the condition and the position there): $(Con1,14,3)$: $y_{22}=0$; $(Con1, 6,10)$: $y_3=0$; $(Con1,1,1)$: $y_{47}=-z_7$; $(Con1, 1,3)$: $y_{46}=0$;
$(Con1, 1,6)$: $y_{40}=-z_3$; $(Con1, 1,7)$ $y_{49}=-z_7+3z_{39}$; $(Con1, 1,9)$: $y_{43}=0$; $(Con1, 1,11)$: $y_{51}=0$; $(Con1, 2,1)$: $z_{15}=-3z_{20}$; $(Con1, 2,3)$: $y_{41}=0$; $(Con1, 2,5)$: $y_5=-3z_{18}$; $(Con1,2,7)$: $z_{44}=-3/2z_{20}$; $Con1, 2,9)$: $y_7=-3z_{24}$; $(Con1,3,1)$: $z_{20}=y_{23}+2z_{35}-z_9$; $(Con1,3,3)$: $z_{41}=0$;
$(Con1, 3,5)$: $z_{18}=z_{11}$; $(Con1, 4,1)$: $y_{23}=-2z_{39}$; $(Con1, 4,2)$: $z_{37}=-y_{24}$;
$(Con1,5,3)$: $ y_{35}=0$; $(Con1,7,5)$: $ z_3=0$;
$(Con1,11,5)$: $ y_9=0$; $(Con1,11,3)$:  $y_{24}=0$; $(Con1,10,3)$: $y_{27}=0$; $(Con1,9,3)$: $ y_{38}=0$;
$(Con1,4,5)$: $z_{11}=0$; $(Con1,5,5)$: $z_2=0$; $(Con2,3,3)$: $z_{33}=1$; $(Con2,14,4)$: $z_{50}=z_{38}$;
$(Con2,11,5)$: $z_{19}=0$; $(Con2,3,13)$: $z_{38}=1-z_{40}$;
$(Con2,10,5)$: $z_{27}=0$; $(Con2,5,5)$: $z_{10}=1$; $(Con2,6,6)$: $z_1=1$; $(Con2,1,9)$: $z_7=0$;
$(Con1,11,6)$: $z_{26}=y_{10}+z_4+y_{33}$;
$(Con2,7,7)$: $z_{43}=1$; $(Con2,9,9)$: $z_{23}=2(z_{21}-1)$; $(Con2,11,12)$: $ z_{31}:=-2z_{29}-z_{24}$;
$(Con3,1,1)$: $ y_1=1$; $(Con3,1,2)$: $y_4=1+2y_2-y_6$;
$(Con3,2,2)$: $y_8=2(y_6-1)$; $(Con3,4,4)$: $y_{30}=1$; $(Con3,14,14)$: $y_{20}=1$; $(Con3,13,13)$: $y_{18}=y_6$;
$(Con3,12,12)$: $y_{11}=1$; $(Con1,13,11)$: $ z_6=0$; $(Con4,2,2)$: $y_{42}=6z_{39}-6z_{35}+3z_9+3z_{10}+3z_4+3y_{33}$;
$(Con2,5,6)$: $z_4=-2z_9$; $(Con2,9,11)$: $z_{21}=1$; $(Con2,12,11)$: $ z_{22}=0$;
$(Con2,10,10)$: $z_{30}=1$; $(Con2,14,7)$: $z_{39}=0$; $(Con2,11,1)$: $y_{33}=2y_{28}-y_{10}+2z_9-z_{17}$;
$(Con3,12,2)$: $y_{13}=y_{10}$; $(Con3,10,6)$: $y_{25}=0$; $(Con3,10,5)$: $y_{26}=0$; $(Con3,7,7)$: $ y_{48}=1$;
$(Con3,3,3)$ $ y_{28}=1$; $(Con3,8,3)$: $ y_{44}=0$; $(Con2,5,9)$: $ z_5=2z_{16}-z_{14}$;
$(Con4,10,5)$: $z_{24}=-2z_{29}$; $(Con3,7,5)$: $y_{39}=1$; $(Con4,5,10)$: $ y_{45}=3z_{14}$;
$(Con4,11,12)$: $z_{16}=-3z_{28}-3z_{25}-z_{13}$; $(Con4,9,6)$: $y_{31}=-2z_9+3z_{17}$;
$(Con4,11,2)$: $ z_{9}=3z_{28}$; $(Con4,5,5)$: $ z_{14}=0$; $(Con2,1,12)$: $z_{8}=0$;
$(Con4,2,1)$: $z_{13}=-3z_{25}$; $(Con4,6,10)$: $y_{52}=-y_{37}+3y_{50}/2+3$;
$(Con4,10,10)$: $ y_{28}=-6z_{28}+3z_{17}$; $(Con1,13,1)$: $ z_{48}=-z_{47}$;
$(Con1,5,4)$: $ y_{37}=3/2+3/2y_6-3z_{12}-3y_2$;
$(Con4,13,13)$: $ y_{17}=3/2y_{19}-z_{45}-z_{46}$;
$(Con1,14,1)$:  $z_{47}=0$; $(Con2,7,4)$: $z_{46}=2z_{42}$; $(Con4,3,3)$: $ y_{19}=-y_{12}$;
$(Con5,3,2)$:  $z_{17}=z_{25}+z_{35}+2z_{28}$; $(Con5,3,3)$: $z_{29}=0$;
$(Con5,3,4)$: $ z_{40}=0$; $(Con5,3,5)$: $z_{35}=-z_{25}$; $(Con5,3,6)$: $z_{25}=-3z_{28}$;
$(Con5,3,13)$: $z_{49}=0$; $(Con5,14,12)$: $z_{32}=1$; $(Con5,11,4)$: $z_{36}=-z_{34}$;
$(Con5,9,4)$: $z_{34}=-1$; $(Con5,3,10)$: $z_{45}=-1/2y_{12}-2z_{42}$;
$(Con1,3,9)$: $ y_{12}=-9z_{28}$; $(Con3,11,2)$: $y_{10}=0$;
$(Con3,7,4)$: $z_{42}=0$; $(Con4,14,14)$: $y_{14}=0$; $(Con4,11,1)$: $z_{12}=1+6z_{28}$;
$(Con4,11,10)$: $y_{29}=-9z_{28}$; $(Con6,7,13)$: $ y_6=1-9/2z_{28}$; $(Con6,9,2)$: $y_{32}=-1+4z_{28}$;
$(Con3,9,5)$: $y_{50}=-2y_{34}-8z_{28}$;
 $(Con7,7,2)$: $y_2=-8y_{34}+7-119/4z_{28}$; $(Con8,13,1)$: $y_{34}=1+137/32z_{28}$;
$(Con8,3,9)$: $y_{36}=1$; $(Con8,13,2)$: $z_{28}=0$.

Thus,  $x_1=x_{\alpha_1}(1)$, $x_2=x_{\alpha_2}(1)$, consequently $\varphi_2(x_{\alpha}(1))=x_{\alpha}(1)$ for any root~$\alpha$.

Now look at the images (under $\varphi_2$) of $h_{\alpha}(t)$, $t\in R^*$.

 Let $h_t=\varphi_2(h_{\alpha_1}(t))$. Since $h_t$ commutes with $h_1,h_2, w_{\alpha_6}(1)$ and $x_{\alpha_6}(1)$, we directly have
{\tiny
$$
h_t=\left(\begin{array}{cccccccccccccc}
d_1& d_2& 0& 0& 0& 0& 0& 0& 0& 0& 0& 0& 0& 0\\
d_3& d_4& 0& 0& 0& 0& 0& 0& 0& 0& 0& 0& 0& 0\\
0& 0& d_5& 0& 0& -d_6& 0& 0& 0& 0& 0& 0& 0& 0\\
0& 0& 0& d_7& 0& 0& -d_8& 0& 0& 0& 0& 0& 0& 0\\
0& 0& 0& 0& d_9& 0& 0& -d_{10}& 0& 0& 0& 0& 0& 0\\
0& 0& 0& 0& 0& d_{11}& 0& 0& -d_{12}& 0& 0& 0& 0& 0\\
0& 0& d_{12}& 0& 0& d_{11}& 0& 0& 0& 0& 0& 0& 0& 0\\
0& 0& 0& d_{10}& 0& 0& d_9& 0& 0& 0& 0& 0& 0& 0\\
0& 0& 0& 0& d_8& 0& 0& d_{7}& 0& 0& 0& 0& 0& 0\\
0& 0& 0& 0& 0& d_{6}& 0& 0& d_{5}& 0& 0& 0& 0& 0\\
0& 0& 0& 0& 0& 0& 0& 0& 0& 0& d_{13}& 0& 0& 0\\
0& 0& 0& 0& 0& 0& 0& 0& 0& 0& 0& d_{13}& 0& 0\\
0& 0& 0& 0& 0& 0& 0& 0& 0& 0& 0& 0& d_{14}& \frac{3}{2}(d_{13}-d_{14})\\
0& 0& 0& 0& 0& 0& 0& 0& 0& 0& 0& 0& 0& d_{13}
\end{array}\right).
$$
}

Use now the conditions $h_{\alpha_1+\alpha_2}(t)=w_2 h_{\alpha_1}(t) w_2^{-1}$, $h_{2\alpha_1+\alpha_2}(t)= w_1 h_{\alpha_1+\alpha_2}(t) w_1^{-1}$ and $h_{\alpha_1}(t)h_{\alpha_1+\alpha_2}(t)=h_{2\alpha_1+\alpha_2}(t)$.
 They give us $w_1w_2h_tw_2^{-1}w_1^{-1}=w_2h_tw_2^{-1}h_t$, therefore $d_2d_{11}=0\Rightarrow d_2=0$; $d_3d_9=0\Rightarrow d_3=0$; $d_5(1-d_{13})=0\Rightarrow d_{13}=1$; $d_6=d_8=d_{10}=d_{12}=0$; $d_7=1/d_5$; $d_1=d_{11}^2$; $d_4=d_9^2$; $d_{11}=1/d_9$; $d_{14}=1$.

 Finally, use the fact that $h_{3\alpha_1+2\alpha_2}(t)=h_{\alpha_1+\alpha_2}(t)h_{2\alpha_1+\alpha_2}(t)$ commutes with $x_1$. It gives us $d_5=d_9^3$, therefore $h_t=h_{\alpha_1}(1/d_9)$.

\section{Final steps of the proof of Theorem 1.}

Now we have stated that for the root system $G_2$
$\varphi_2(x_{\alpha}(1))=x_{\alpha}(1)$, $\varphi_2(h_{\alpha}(t))=h_{\alpha}(s)$, $\alpha\in \Phi$,
$t\in R^*$, $s\in S^*$.

For every long root $\alpha_j$ there exists a root $\alpha_k$ such that
$h_{\alpha_k}(t)x_{\alpha_j}(1)h_{\alpha_k}(t)^{-1}=x_{\alpha_j}(t)$. Therefore, $\varphi_2(x_{\alpha_j}(t))
=x_{\alpha_j}(s)$. From the conditions written above and commutator formulas it follows
$\varphi_2(x_{\alpha}(t))=x_{\alpha}(s)$ for all $\alpha\in \Phi$.

Denote the mapping $t\mapsto s$ by $\rho: R^* \to S^*$. If  $t\notin R^*$, then $t\in J$, i.\,e., $t=1+t_1$, where
$t_1\in R^*$. Then $\varphi_2(x_\alpha(t))=\varphi_2(x_\alpha(1)x_\alpha(t_1))=x_\alpha(1)x_\alpha(\rho(t_1))=
x_\alpha(1+\rho(t_1))$, $\alpha\in \Phi$. Therefore we can continue the mapping $\rho$ on the whole ring~$R$ (by the formula
$\rho(t):=1+\rho(t-1)$ for $t\in J$), and obtain $\varphi_2(x_\alpha(t))=x_\alpha(\rho(t))$ for all $t\in R$, $\alpha\in
\Phi$.
 Clear that $\rho$ is injective, additive, multiplicative on invertible elements. Since every element of~$R$
is a sum of two invertible elements, we have that $\rho$ is multiplicative on the whole~$R$, i.\,e., it is an isomorphism of~$R$ onto some subring in~$S$.
Note that in this situation $C G(S)
C^{-1}=G(S')$ for some matrix $C\in \GL_{14}(S)$. Show that $S'=S$.

\begin{lemma}\label{porozhd} The Chevalley group $G(S)$ generates  $M_{14}(S)$ as a ring.
\end{lemma}
\begin{proof}
The matrix
$\frac{1}{2}(x_{\alpha_2}(1)-1)^2$ is $e_{3,4}$ ($e_{i,j}$ is a matrix unit). Multiplying it to the suitable diagonal matrix we can obtain an arbitrary matrix of the form $\alpha\cdot e_{3,4}$ (since the invertible elements of~$S$
generate~$S$).

According to the fact that all long root are conjugate under the action of the Weil group, multiplying $e_{3,4}$ from left and right sides to suitable  $w_i,w_j$, we obtain all $e_{k,l}$, $k,l=3,4,9,10,11,12$.

 Then $e_{14,4}=(x_{\alpha_2}(1)-E)\cdot e_{4,4}+e_{3,4}$, again with the help of multiplying (from the right side) to different elements of the Weil group we obtain all $e_{14,l}$, $l=3,4,9,10,11,12$.

 Now $e_{14,14}=-(e_{14,3}(x_{\alpha_2}(1)-E)+e_{14,4})(w_1-E)$; $e_{3,14}=-1/2 (x_{\alpha_2}(1)-E)e_{14,14}$. As above, we obtain all $e_{l,14}$, $l=3,4,9,10,11,12$.

 Also we have $e_{14,13}=e_{14,3}(x_{\alpha_2}(1)-E)+e_{14,4}+2e_{14,14}$; $e_{3,13}=e_{3,3}(x_{\alpha_2}(1)-E)+2e_{3,14}+e_{3,4}$ (and directly obtain all $e_{l,13}$, $l=3,4,9,10,11,12$);
$e_{3,2}=e_{3,13}(x_{\alpha_1}(1)-E)$. Now use the fact that all short roots are also conjugate under the action of~$W$; we obtain all $e_{k,l}$, $k=3,4,9,10,11,12$, $l=1,2,5,6,7,8$.

 Since $e_{13,13}=1/4(h_1+E)(h_2+E)-e_{14,14}$,  $e_{1,13}=-1/2(x_{\alpha_1}-E)e_{13,13}$, $e_{13,2}=e_{13,13}(x_{\alpha_1}-E)$, we have all $e_{l,13}, e_{13,l}$, $l=1,2,5,6,7,8$, and after that by multiplying matrix units we obtain all $e_{l,k}$, $l,k=1,\dots,14$, and therefore the whole ring $M_{14}(S)$.

Lemma is proved.
 \end{proof}

Note that in the previous lemma we did  not suppose $1/3\in S$.

\begin{lemma}\label{Tema}
If for some $C\in \GL_{14}(S)$ we have $C G(S) C^{-1}= G(S')$, where $S'$ is a subring of~$S$, then $S'=S$. \end{lemma}
\begin{proof}
Suppose that $S'$ is a proper subring of~$S$.

Then $C M_{14}(S) C^{-1} =M_{14} (S')$, since the group $G(S)$ generates $M_{14}(S)$, and the group $G(S')=CG(S)
C^{-1}$ generates $M_{14}(S')$. It is impossible, since $C\in \GL_{14}(S)$.
\end{proof}

Therefore we have proved that   $\rho$ is an isomorphism from the ring~$R$ onto~$S$. Consequently the composition of the initial isomorphism~$\varphi'$ and some basis change with a matrix $C\in \GL_{14}(S)$, (mapping $G(S)$ onto itself),
is a ring isomorphism~$\rho$. Therefore, $\varphi' = i_{C^{-1}} \circ \rho$.

So Theorem 1 is proved.

\section{Proof of Theorem 2.}

In this section we still consider a Chevalley group $G(R)$ of type $G_2$ (it is simultaneously elementary and adjoint) over a local ring with~$1/2$ and $1/3$.

\begin{theorem}
The normalizer of $G(R)$ in $\GL_{14}(R)$ is $\lambda\cdot G(R)$.
\end{theorem}

\begin{proof}
Suppose that we have some matrix $C=(c_{i,j})\in \GL_{14}(R)$ such that
$$
C\cdot G(R) \cdot C^{-1}=G(R).
$$

If $J$ is the radical of $R$, then the matrices from
$M_{14}(J)$ is the radical in the matrix ring $M_{14}(R)$, therefore
$$
C\cdot M_{14}(J)\cdot C^{-1}=M_{14}(J),
$$
consequently,
$$
C\cdot (E+M_{14}(J))\cdot C^{-1}=E+M_{14}(J),
$$
i.\,e.,
$$
C\cdot G(R,J)\cdot C^{-1}=G(R,J),
$$
since $G(R,J)=G(R)\cap (E+M_{14}(J)).$

Thus, the image $\overline C$ of the matrix~$C$
under factorization~$R$ by~$J$ gives us an
automorphism--conjugation of the Chevalley group $G(k)$, where $k=R/J$ is a residue field of~$R$.

\begin{lemma}\label{fromAnton}
If $G(k)$ is a Chevalley group of type $G_2$ over a field $k$ of characteristics $\ne 3$, then every its automorphism--conjugation  is inner.
\end{lemma}
\begin{proof}
By Theorem~30 from~\cite{Steinberg} every automorphism of a Chevalley group of type~$G_2$ over a field~$k$ of characteristics $\ne 3$ is standard, i.\,e., for this type it ia a composition of inner and ring automorphisms. Suppose that a matrix~$C$ is from normalizer of $G(k)$ in $\GL_{14}(k)$. Then $i_C$ is an automorphism of $G(k)$, so we have $i_C=i_g\circ \rho$, $g\in G(k)$, $\rho$ is a ring automorphism. Consequently, $i_{g^{-1}}i_C=i_{C'}=\rho$ and some matrix $C'\in \GL_{14}(k)$ defines a ring automorphism~$\rho$. For every root $\alpha\in \Phi$ we have $\rho(x_\alpha(1))=x_\alpha(1)$, therefore $C'x_\alpha(1)=x_\alpha(1) C'$ for all $\alpha\in \Phi$. Thus we have that $C'$ is scalar and an automorphism $i_C$ is inner.
\end{proof}  

By Lemma~\ref{fromAnton}
$$
i_{\overline C}=i_g,\quad  g\in G(k).
$$

Since over a field every element of a Chevalley group is a product of some set of unipotents $x_\alpha(t)$) and the matrix  $g$ can be decomposed into a product $x_{\alpha_{i_1}}(Y_1)\dots x_{i_{N}}(Y_{N})$, где
$Y_1,\dots, Y_{N}\in k$.

Since every element $Y_1,\dots, Y_N$
is a residue class in~$R$, we can choose
(arbitrarily) elements $y_1\in Y_1$, \dots, $y_N\in Y_N$, and the element
$$
g'=x_{\alpha_{i_1}}(y_1)\dots x_{i_N}(y_N)
$$
satisfies the conditions $g'\in G(R)$ and $\overline
{g'}=g$.

Consider the matrix $C'={g'}^{-1}\circ d^{-1}\circ C$. This matrix also normalizes the group $G(R)$, and
also $\overline {C'}=E$. Therefore, from the description of the normalizer of~$G(R)$ we come to the description of all matrices from this normalizer equivalent to the unit matrix modulo~$J$.

Therefore we can suppose that our initial matrix $C$ is equivalent to the unit modulo~$J$.

Our aim is to show that $C\in \lambda G(R)$.

Firstly we prove one technical lemma that we will need later.

\begin{lemma}\label{prod2}
 Let $X=\lambda t_{\alpha_1}(s_1)t_{\alpha_2}(s_2)x_{\alpha_1}(t_1)\dots x_{\alpha_6}(t_6)x_{-\alpha_1}(u_1)\dots x_{-\alpha_6}(u_6)\in \lambda G(R,J)$.
Then the matrix $X$ contains $15$ coefficients \emph{(}precisely described in the proof of lemma\emph{)}, uniquely defining all $\lambda$, $s_1,s_2$, $t_1,\dots,t_6$, $u_1,\dots, u_6$.
\end{lemma}

\begin{proof}
By direct calculation we obtain that in the matrix~$X$
$x_{12,12}=\frac{\lambda }{s_1^3s_2^2}$, $x_{12,10}=-\frac{\lambda u_2}{s_1^3s_2^2}$, therefore we find $u_2$;
$x_{12,8}=\frac{\lambda u_3}{s_1^3s_2^2}$, so we get  $u_3$;
similarly from $x_{12,6}=\frac{\lambda u_4}{s_1^3s_2^2}$,
$x_{12,4}=-\frac{\lambda u_5}{s_1^3s_2^2}$ и $x_{12,14}=-\frac{\lambda(u_6+u_2u_5)}{s_1^3s_2^2}$ we know $u_4,u_5,u_6$. Besides, from $x_{10,12}=\frac{\lambda t_2}{s_1^3 s_2}$ and $x_{10,10}=\frac{\lambda (1-t_2 u_2)}{s_1^3s_2}$ we find
$t_2$ and $\frac{\lambda}{s_1^3 s_2}$, and consequently we know~$s_2$.

 From $x_{10,8}=-\frac{\lambda (u_1-t_2u_3}{s_1^3 s_2}$ we find~$u_1$. From the first equation we can express $\lambda$ by $s_1$. Therefore, $\lambda$ is now known.
Now from the system of equations
\begin{align*}
x_{14,12}&=\lambda(t_2t_5+3t_3t_4+2t_6);\\
x_{4,12}&=\frac{\lambda(t_5+3t_1t_4+3t_1^2t_3-t_1^3t_2)}{s_2};\\
x_{4,6}&=\frac{-t_1+2t_1^2u_1-t_1^3u_1^2+u_4t_5+3u_4t_1t_4+3u_4t_1^2t_3-u_4t_1^3t_2)}{s_2};\\
x_{8,8}&=\frac{\lambda (1-3t_1u_1-3u_3t_3+3u_3t_1t_2)}{s_1^2s_2};\\
x_{14,6}&=\lambda(-t_3-2t_4u_1-t_5u_1^2+u_4t_2t_5+3u_4t_3t_4+2u_4t_6);\\
x_{14,8}&=\lambda(t_4+t_5u_1+u_3t_2t_5+3u_3t_3t_4+2u_3t_6),
\end{align*}
 where $s_1, t_1,t_3,t_4,t_5,t_6$ are variables, in every equation there is exactly one variable with invertible coefficient, and for all equations these variables are different, we can find all six variables.

Now we know all obtained elements of the ring. Lemma is proved.

\end{proof}

Now return to our main proof.
Recall, that we work with the matrix~$C$, equivalent to the unit matrix modulo radical, and mapping the Chevalley group into itself.

For every root $\alpha\in\Phi$ we have
\begin{equation}\label{osn_eq}
C x_{\alpha}(1)C^{-1}=x_{\alpha}(1)\cdot g_\alpha,\quad g_\alpha\in
G(R,J).
\end{equation}
Every $g_\alpha \in G(R,J)$ can be decomposed into a product
\begin{equation}\label{razl_rad}
 t_{\alpha_1}(1+a_1)
t_{\alpha_2}(1+a_2)x_{\alpha_1}(b_1)\dots
x_{\alpha_6}(b_6)x_{\alpha_{-1}}(c_1)\dots x_{-\alpha_{6}}(c_6),
\end{equation}
where $a_1,a_2,b_1,\dots,b_6,c_1,\dots, c_6\in J$ (see,
for example,~\cite{Abe1}).

Let $C=E+X=E+(x_{i,j})$. Then for every root~$\alpha\in \Phi$
we can write a matrix equation~\ref{osn_eq} with variables $x_{i,j},
a_1,a_2,b_1,\dots,b_6,c_1,\dots, c_6$, each of them is from the radical.

Let us change these equations.
We consider the matrix~$C$ and ``imagine'', that it is some matrix from Lemma~\ref{prod2} (i.\,e., it is from $\lambda G(R)$). Then by some its concrete $15$~positions we can  ``define'' all coefficients $\lambda, s_1, s_2, t_1,\dots,t_{6},u_1,\dots, u_{6}$ in the decomposition of this matrix from Lemma~\ref{prod2}. In the result we obtain a matrix $D\in \lambda G(R)$, every matrix coefficient in it is some (known) function of coefficients of~$C$.
Change now the equations~\eqref{osn_eq} to the equations
\begin{equation}\label{fol_eq}
D^{-1}C x_{\alpha}(1)C^{-1}D=x_{\alpha}(1)\cdot {g_\alpha}',\quad {g_\alpha}'\in
G(R,J).
\end{equation}
We again have matrix equations, but with variables $y_{i,j},
a_1',a_2',b_1',\dots,b_6',c_1',\dots, c_6'$, every of them still is from radical, and also every
 $y_{p,q}$ is some known function of (all) $x_{i,j}$. The matrix $D^{-1}C$ will be denoted by~$C'$.

We want to show that a solution exists only for all variables with  primes equal to zero. Some $x_{i,j}$ also will equal to zero, and other are reduced in the equations. Since the equations are very complicated we will consider the linearized system. It is sufficient to show that all variables from the linearized system (let it be the system of  $q$~variables) are members of some system from $q$ linear equations with invertible in~$R$ determinant.

In other words, from the matrix equalities we will show that all variables from them are equal to zeros.

Clear that  linearizing  the product $Y^{-1}(E+X)$ we obtain some matrix $E+(z_{i,j})$, with all positions described in Lemma~\ref{prod2} equal to zero.

To find a final form of a linearized system, we write the last one as:
\begin{multline*}
(E+Z)x_\alpha(1) =x_\alpha(1)(E+a_1T_1+a_1^2\dots)
(E+a_2T_2+a_2^2\dots)\cdot\\
\cdot(E+b_1X_{\alpha_1}+b_1^2X_{\alpha_1}^2/2+\dots)\dots
(E+c_6X_{-\alpha_6}+c_6^2X_{-\alpha_6}^2/2+\dots)(E+Z),
\end{multline*}
where $X_\alpha$ is a corresponding element of the Lie algebra in its adjoint representation,
\begin{align*}
T_1&=\diag[1,1,0,0,1,1,2,2,3,3,3,3,0,0];\\
T_2&=\diag[0,0,1,1,1,1,1,1,1,1,2,2,0,0].
\end{align*}

Finally we have
$$
Zx_{\alpha}(1)-x_{\alpha}(1)(Z+a_1T_1+a_2T_2+b_1X_{\alpha_1}+\dots+c_6X_{\alpha_6})=0.
$$
This equation can be written for every $\alpha\in
\Phi$ (naturally, with another $a_j, b_j, c_j$), and can be written only for generating roots: for $\alpha_1,
\alpha_2, -\alpha_1,  -\alpha_2$. The number of free variables is not changed.

We have four equations:
$$
\begin{cases}
Zx_{\alpha_1}(1)-x_{\alpha_1}(1)(X+a_{1,1}T_1+a_{2,1}T_2+\\
\ \ \ \ \
+b_{1,1}X_{\alpha_1}+\dots+b_{6,1}X_{\alpha_6}
+c_{1,1}X_{-\alpha_1}+\dots+c_{6,1}X_{-\alpha_6})=0;\\
Zx_{\alpha_2}(1)-x_{\alpha_2}(1)(X+a_{1,2}T_1+a_{2,2}T_2+\\
\ \ \ \ \ +b_{1,2}X_{\alpha_1}+\dots+b_{6,2}X_{\alpha_6}
+c_{1,2}X_{-\alpha_1}+c_{6,2}X_{-\alpha_6})=0;\\
Xx_{-\alpha_1}(1)-x_{-\alpha_1}(1)(X+a_{1,3}T_1+a_{2,3}T_2+\\
\ \ \ \ \
+b_{1,3}X_{\alpha_1}+\dots+b_{6,3}X_{\alpha_6}
+c_{1,3}X_{-\alpha_1}+\dots+c_{6,3}X_{-\alpha_6})=0;\\
Xx_{-\alpha_2}(1)-x_{-\alpha_2}(1)(X+a_{1,4}T_1+a_{2,4}T_2+\\
\ \ \ \ \
+b_{1,4}X_{\alpha_1}+\dots+b_{6,4}X_{\alpha_6}
+c_{1,4}X_{-\alpha_1}+\dots+c_{6,4}X_{-\alpha_6})=0.
\end{cases}
$$

 The matrix $Z=(z_{i,j})$ has zero elements on the positions  $z_{4,6}$, $z_{4,12}$, $z_{8,8}$, $z_{10,8}$, $z_{10,10}$, $z_{10,12}$, $z_{12,4}$, $z_{12,6}$,
$z_{12,8}$, $z_{12,10}$, $z_{12,12}$, $z_{12,14}$, $z_{14,6}$, $z_{14,8}$, $z_{14,12}$.

From the matrix of the second condition we have: the position $(3,6)$: $z_{3,4}=0$; the position $(3,7)$: $z_{3,9}=0$; the position $(3,11)$: $c_{5,1}=0$; the position $(10,6)$: $z_{10,4}=0$; the position $(10,7)$: $z_{10,9}=0$; the position $(10,12)$: $b_{2,1}=0$; the position $(11,6)$: $z_{11,4}=0$; the position $(11,13)$: $z_{11,1}=0$; the position $(2,6)$: $z_{2,4}=0$; the position $(2,7)$: $z_{2,9}=0$; the position $(3,8)$: $z_{3,6}=0$; the position $(3,13)$: $z_{3,1}=0$; the position $(5,6)$: $z_{5,4}=0$; the position $(6,9)$: $z_{8,9}=0$; the position $(11,8)$: $z_{11,6}=0$;  the position $(11,2)$: $z_{11,13}=0$; the position $(12,8)$: $c_{3,1}=0$; the position $(12,6)$: $c_{4,1}=0$; the position $(12,10)$: $c_{2,1}=0$; the position $(14,11)$: $c_{6,1}=0$; the position $(14,12)$: $b_{6,1}=0$; the position $(10,13)$: $z_{10,1}=0$; the position $(12,7)$: $z_{12,9}=0$; the position $(11,8)$: $z_{11,6}=0$; the position $(10,2)$: $z_{10,13}=0$; the position $(13,12)$: $z_{2,12}=0$; the position $(13,11)$: $z_{2,11}=0$; the position $(12,5)$: $z_{12,7}=0$; the position $(6,1)$: $z_{8,1}=0$; the position $(14,7)$: $z_{14,9}=0$; the position $(12,13)$: $z_{12,1}=0$; the position $(14,13)$: $z_{14,1}=0$; the position $(14,5)$: $z_{14,7}=0$; the position $(12,2)$: $z_{12,13}=0$; the position $(1,9)$: $z_{13,9}=0$; the position $(4,9)$: $z_{6,9}=0$; the position $(14,2)$: $z_{14,13}=0$; the position $(2,5)$: $z_{2,7}=0$; the position $(1,12)$: $z_{13,12}=0$; the position $(10,5)$: $z_{10,7}=0$; the position $(9,4)$: $z_{7,4}=0$.

From the matrix of the second condition it follows:  the positions $(9,6)$: $z_{9,2}=0$;  $(13,13)$: $z_{13,3}=0$; $(12,13)$: $z_{12,3}=0$; $(12,14)$: $с_{6,2}=0$; $(14,11)$: $z_{4,11}=0$; $(12,9)$: $z_{12,11}=0$; $(13,5)$: $с_{3,2}=0$; $(11,10)$: $z_{9,10}=0$; $(8,9)$: $z_{8,11}=0$; $(11,8)$: $z_{9,8}=0$; $(10,9)$: $z_{10,11}=0$; $(6,9)$: $z_{6,11}=0$; $(2,3)$: $z_{6,3}=0$; $(4,1)$: $z_{4,5}=0$; $(10,1)$: $z_{10,5}=0$; $(10,5)$: $z_{12,5}=0$; $(7,12)$: $z_{7,10}=0$; $(5,10)$: $z_{1,10}=0$; $(7,5)$: $b_{1,2}=0$; $(12,10)$: $c_{2,2}=0$; $(5,3)$: $z_{1,3}=0$; $(13,7)$: $c_{4,2}=0$; $(13,8)$: $b_{4,2}=0$;   $(10,8)$: $c_{1,2}=0$; $(12,4)$: $c_{5,2}=0$; $(7,2)$: $b_{5,2}=0$; $(1,6)$: $z_{1,2}=0$; $(4,12)$: $z_{4,10}=0$; $(5,11)$: $z_{1,11}=0$; $(13,9)$: $z_{13,11}=0$; $(4,6)$: $z_{4,2}=0$;  $(7,9)$: $z_{7,11}=0$; $(10,6)$: $z_{10,2}=0$;
$(14,8)$: $z_{4,8}=0$; $(4,13)$: $z_{4,3}=0$; $(7,13)$: $z_{7,3}=0$; $(8,13)$: $z_{8,3}=0$; $(9,13)$: $z_{9,3}=0$; $(10,13)$: $z_{10,3}=0$; $(6,1)$: $z_{6,5}=0$; $(8,1)$: $z_{8,5}=0$; $(9,1)$: $z_{9,5}=0$; $(13,1)$: $z_{13,5}=0$; $(10,2)$: $z_{12,2}=0$; $(6,4)$: $z_{6,14}=0$; $(7,4)$: $z_{7,14}=0$; $(8,4)$: $z_{8,14}=0$; $(10,4)$: $z_{10,14}=0$; $(13,4)$: $z_{13,14}=0$; $(4,5)$: $z_{14,5}=0$; $(8,6)$: $z_{8,2}=0$; $(2,7)$: $z_{6,7}=0$; $(5,7)$: $z_{1,7}=0$; $(14,7)$: $z_{4,7}=0$; $(3,7)$: $z_{13,7}=0$; $(3,8)$: $z_{13,8}=0$; $(8,8)$ и $(10,10)$: $a_{1,2}=0$ и $a_{2,2}=0$; $(1,1)$: $z_{1,5}=0$; $(2,2)$: $z_{6,2}=0$; $(3,3)$: $z_{14,3}=0$; $(4,4)$: $z_{4,14}=0$; $(9,9)$: $z_{9,11}=0$; $(10,12)$: $b_{2,2}=0$; $(3,11)$: $z_{14,11}=0$; $(2,10)$: $z_{6,10}=0$; $(2,14)$: $z_{2,3}=0$; $(2,13)$: $z_{6,13}=0$; $(14,13)$: $z_{4,13}=0$; $(14,1)$: $z_{4,1}=0$; $(14,9)$: $z_{4,9}=0$; $(14,4)$: $z_{14,14}=0$; $(3,9)$: $z_{3,11}=0$; $(3,14)$: $z_{13,13}=0$; $(3,13)$: $z_{3,3}=0$.

Again return to the first condition. Now the elements $z_{1,9}$, $z_{2,5}$, $z_{6,1}$, $z_{8,7}$, $z_{8,13}$, $z_{5,6}$, $z_{5,11}$ become to be zeros.

If we consider the condition~3, we have that from all coefficients $a_i,b_i,c_i$ only $b_{2,3},b_{3,3}, b_{4,3}, c_{1,3}, a_{1,3}$ can be not zeros. The following coefficients of~$Z$ are zeros: $z_{5,3}$, $z_{5,10}$, $z_{1,12}$, $z_{7,8}$, $z_{11,2}$, $z_{11,10}$, $z_{11,8}$, $z_{14,2}$.

Now in the second condition the position $(14,6)$ gives $b_{3,2}=0$, therefore $z_{13,2}$, $z_{6,8}$, $z_{1,14}$, $z_{7,5}$, $z_{1,13}$, $z_{8,10}$, $z_{9,6}$, $z_{9,7}$, $z_{9,12}$ are zeros.

Again from the third condition $a_{1,3}=0$, $z_{7,12}=0$.

Finally, come to the last, fourth condition. Since the roots $\alpha_2$ and $-\alpha_2$ are conjugate with the element~$w_2$, it is clear that in the fourth condition from all $a_i,b_i,c_i$ only  $b_{5,4}$ can be not zero. Therefore the following elements of~$Z$ are zeros: $z_{1,4}$, $z_{5,7}$, $z_{5,9}$, $z_{5,12}$, $z_{5,1}$, $z_{5,13}$, $z_{5,14}$, $z_{2,1}$, $z_{2,6}$, $z_{2,14}$, $z_{3,2}$, $z_{3,12}$, $z_{3,14}$, $z_{3,5}$, $z_{3,7}$, $z_{3,8}$, $z_{3,13}$, $z_{13,1}$, $z_{13,6}$, $z_{6,4}$, $z_{2,8}$, $z_{2,13}$, $z_{7,1}$, $z_{7,6}$, $z_{7,9}$, $z_{8,4}$, $z_{8,6}$, $z_{8,12}$, $z_{9,4}$, $z_{10,6}$, $z_{11,7}$, $z_{11,9}$, $z_{11,12}$, $z_{11,14}$, $z_{3,10}$, $z_{13,4}$, $z_{14,4}$, $z_{5,8}$.

From the first equation now $z_{2,10}=z_{6,12}=z_{1,6}=0$, from the second one $z_{1,8}=z_{5,2}=z_{7,2}=z_{13,10}=z_{14,10}=z_{11,3}=z_{9,13}=z_{9,14}=0$, from the third one $z_{9,1}=z_{11,5}=z_{7,13}=0$. Again from the second condition it follows  $z_{6,6}=z_{2,2}$, $z_{5,5}=z_{1,1}$, $z_{11,11}=z_{9,9}$. From the first condition  $z_{4,4}=z_{1,1}=z_{2,2}=z_{7,7}=z_{9,9}=0$. Therefore $Z=0$, what we need.

Theorem~2 is proved.
\end{proof}

From Theorems 1 and 2 it directly follows the main theorem of this paper:

\begin{theorem}\label{main}
 Suppose that $G(R)$ and $G(S)$ are Chevalley groups of type  $G_2$,  $R$, $S$ are local rings with~$1/2$ and $1/3$. Then every isomorphism between~$G(R)$ and $G(S)$ is standard, i.\,e., it is the composition of  a ring isomorphism and an inner automorphism.
\end{theorem}

\end{document}